\documentclass{amsart}
\usepackage{graphicx} 
\usepackage{subfigure}
\usepackage{subcaption}
\usepackage{amsmath}
\usepackage{amsthm}
\usepackage{amssymb}
\usepackage{natbib}
\usepackage{amsfonts}
\usepackage{titlesec}
\titlelabel{\thetitle.\quad}
\usepackage{caption}
\usepackage{float}
\usepackage{color,xcolor,mathrsfs,mathtools,calc,bm,array,graphicx}
\usepackage{multimedia,media9,xmpmulti,siunitx,animate,tikz}

\newtheorem{theorem}{Theorem}
\newtheorem{lemma}{Lemma}

\theoremstyle{definition} 
\newtheorem{example}{Example}

\newtheorem{definition}{Definition}
\newtheorem{remark}{Remark}

\title[Harmonic Functions with Interesting Caustics]{Zeros of Harmonic Functions whose Caustic is a Non-Singular Image of an Epicycloid}
\author{Eli Sampson}
\date{August 2025}

\begin{document}

\maketitle

\begin{abstract}
    Recent researchers have investigated how the zeros of certain families of complex harmonic functions change with a single parameter. Many leverage the well-behaved images of the critical curve and the harmonic analogue of the Argument Principle to prove zero-counting theorems. In this paper, we investigate the zeros of a family of harmonic functions for which the image of its critical curve is a non-singular linear image of an epicycloid. By analyzing this curve and using the harmonic analogue of the Argument Principle, we obtain a detailed zero-counting theorem for our family.
\end{abstract}

\section{Introduction}

We consider the zeros of a family of complex-valued harmonic functions. Such functions have the form $f = u+iv$, where both $u$ and $v$ are real-valued harmonic functions. They can also be written in the form $f = h+\overline{g}$, where $h$ and $g$ are analytic.

Extensive research has explored the similarities and differences between analytic and harmonic functions.
Analytic functions on a simply connected domain are sense-preserving (or orientation-preserving) at all points. Complex-valued harmonic functions, however, may have regions where they are sense-preserving and regions where they are sense-reversing. The Fundamental Theorem of Algebra states that for non-constant analytic polynomials, the number of zeros is the same as the degree. For harmonic polynomials, however, the correct generalization is more nuanced and depends on how many zeros are located in both the sense-preserving and sense-reversing regions.

Researchers have worked extensively to understand how many zeros a harmonic function can have. Sheil-Small \cite{sheil} conjectured and Wilmhurst \cite{wilmshurst1998valence} proved that for a harmonic polynomial $f = h + \overline{g}$ for which $\deg h = n$, $\deg g= m$, and $m < n$, the largest number of zeros $f$ can have is $n^2$. Building upon this result,  S\`{e}te and Zur \cite{sete2024zeros} showed that for every $k  = n,n + 1,...,n^2$ there exists a harmonic polynomial of degree $n$ with $k$ zeros.

Subsequent work has focused on proving bounds on the number of zeros for specific families of harmonic functions and analyzing the locations of their zeros. (See \cite{melman24}, \cite{legesse2022location}, \cite{melman1}, \cite{alemu2022imagecriticalcirclezerofree}.) Others seek to prove detailed zero-counting theorems that explain precisely how the number of zeros changes with certain parameters. (See \cite{BDHPWW}, \cite{BDDLO}, \cite{sam}, \cite{Ararso_Alemu_2022}, \cite{brooks2025usingrealvariabletechniquesstudy}.)

For example, Brilleslyper et al. \cite{BBDHS} investigated a family of trinomials
\begin{equation}\label{Brilleslyper}
    p_c(z) =  z^n + c\overline{z}^k -1
\end{equation}
where $1 \leq k \leq n-1$, $n\geq3,$ $c \in \mathbb{R}_{+}$, and $\gcd(n,k) = 1$. 
They concluded that as $c$ increases, the number of zeros increases monotonically from a minimum of $n$ to a maximum of $n + 2k$. Considering a similar family but with poles, Brooks and Lee \cite{lee} investigated the family 
\begin{equation}\label{leethesis}
r_c(z) = z^n + \frac{c}{\overline{z}^k} - 1, \hspace{1em} c \in \mathbb{C} \setminus \{0\}
\end{equation}
for $n,k \in \mathbb{N}$ and proved a similarly detailed zero-counting theorem. The strategies they implemented to analyze both families were similar. First, they leveraged the fact that for both families the critical curve (the curve that separates the sense-preserving and sense-reversing regions) is a circle. Then, they used the fact that the caustic (the image of the critical curve under the function) is a well-known parametric curve. For the family \eqref{Brilleslyper}, the caustic is a hypocycloid and in the case of the family \eqref{leethesis}, it is an epicycloid. Then, using the harmonic analogue of the Argument Principle, they counted the zeros by finding the winding number about the origin of the caustic.

While this method of analysis is powerful for proving detailed zero-counting theorems, it is limited to functions whose caustic is an easily analyzable parametric curve. In this paper, we explore how a similar strategy can be used to count the zeros of a family of functions whose caustic is a nonstandard parametric curve. Specifically, we investigate the family
\begin{equation}\label{ourfamily}
 f_a(z) = \frac{a}{n+1}z^{n+1} - \frac{1}{n}z^{-n} + \frac{1}{n+1}\overline{z}^{n+1} - \frac{a}{n}\overline{z}^{-n}-1   
\end{equation}
with $n \in \mathbb{N}$, $n\geq4$, and $a>1$. By design, this family of functions has a critical curve that is a circle and a caustic that is the non-singular image of a standard parametric curve. By analyzing the transformation between the well-understood standard parametric curve and the caustic of our function, we find the winding number of the caustic and then apply the harmonic analogue of the Argument Principle to prove a detailed zero-counting theorem for the family \eqref{ourfamily}. 

\begin{theorem}\label{maintheorem}
    Let $f_a$ be as in \eqref{ourfamily} and let $N = \lfloor \frac{n+1}{2} \rfloor$. There exist $N$ critical values $a_j$ with $1 < a_1 < a_2 < \cdots < a_N$, such that:

    \begin{enumerate}
        \item if $1 < a < a_1$, then $f_a$ has $2n+1$ zeros.
        \item if $a_j < a < a_{j+1}$ for some $1 \leq j < N$, then
            \begin{enumerate}
                \item if $n$ is even, then $f_a$ has $2n - 4j + 1$ zeros.
                \item if $n$ is odd, then $f_a$ has $2n - 4j + 3$ zeros.
            \end{enumerate}
        \item if $a_N < a$, then $f_a$ has $1$ zero.
    \end{enumerate}
\end{theorem}
The following example and figure illustrate how to use Theorem \ref{maintheorem}.
\begin{example}
Consider $n=4$, so that \begin{equation*} \label{examplefunction}
f_a(z) = \frac{a}{5}z^{5} - \frac{1}{4}z^{-4} + \frac{1}{5}\overline{z}^{5} - \frac{a}{4}\overline{z}^{-4} - 1.
\end{equation*}
By Theorem \ref{maintheorem}, there are $N = 2$ critical values of $a$. Figure~\ref{zeros_fig} shows that when \( a = 1.1 \), the function \( f_a \) has \( 2(4) + 1 = 9 \) zeros. Increasing \( a \) to \( a = 1.37 \), where \( a_1 < 1 < a_2 \), reduces the number of zeros to \( 2(4) - 4 + 1 = 5 \). Further increasing \( a \) to \( a = 3.54 \), where \( a_2 < a \), results in just 1 zero.

\end{example}

\begin{figure}[H]
  \subfigure[$a = 1.1$]{
    \includegraphics[scale = 0.28]{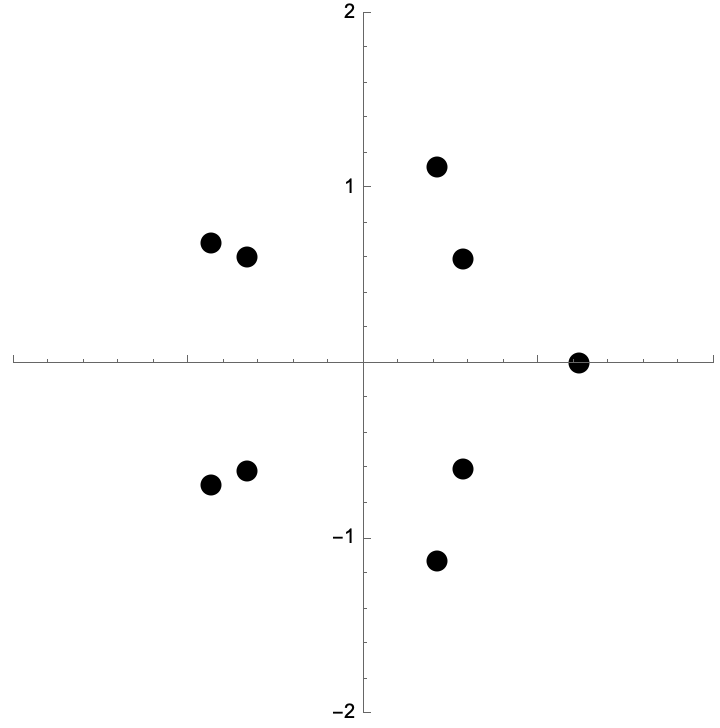}
  }
  \hfill  
  \subfigure[$a = 1.37$]{
    \includegraphics[scale = 0.28]{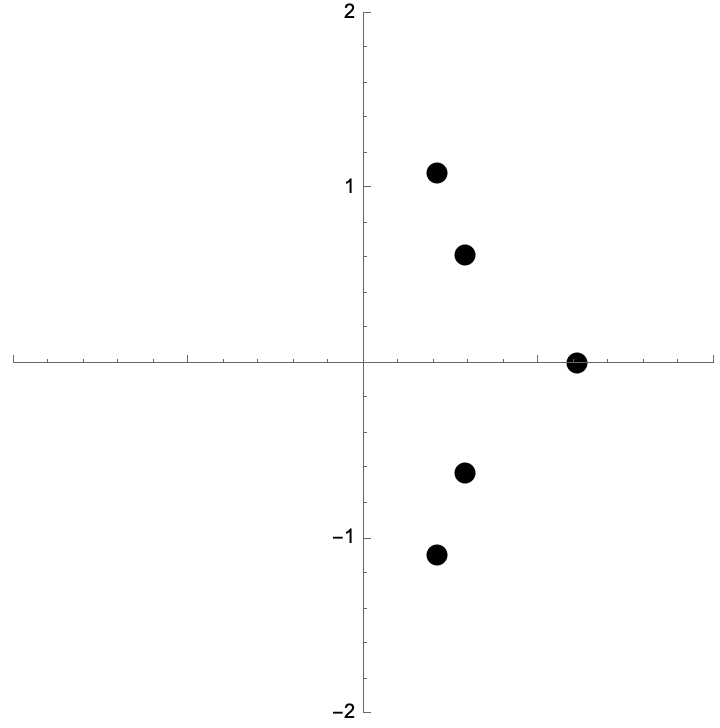}
  }
  \hfill  
  \subfigure[$a = 3.54$]{
    \includegraphics[scale = 0.28]{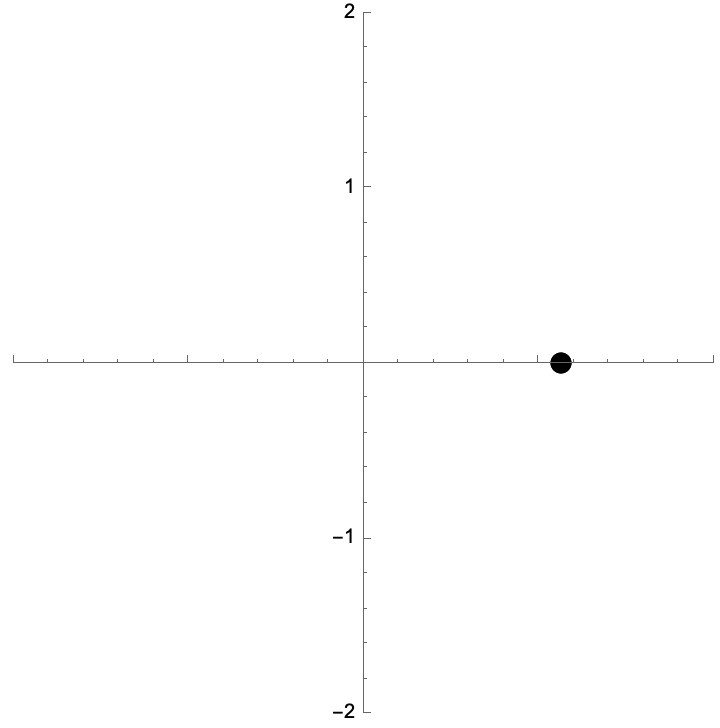}
  }
  \caption{The zeros of $f_a$ for $n=4$. The total number of zeros ranges between $9$ and $1$.}\label{zeros_fig}
\end{figure}

The remainder of this paper is organized as follows: In Section \ref{background}, we provide key definitions and theorems. In Section \ref{ourcriticalcurve}, we identify the parameterization of the image of the critical curve of $f_a$. Section \ref{Formula For Number of Zeros} contains a lemma that relates the total number of zeros to the winding number of the image of the critical curve. Then, in Section \ref{geometry section}, we prove lemmas about the dependence of the winding number of the image of the critical curve on the parameter $a$. Lastly, in Section \ref{endproof}, we use these lemmas to prove Theorem \ref{maintheorem}.

I would like to thank Jennifer Brooks for her support and guidance throughout this project. Portions of Section \ref{background} are adapted from the work of Brooks et al. \cite{undergradbrooks}, which served as a helpful model for this investigation.

\section{Background}\label{background}

A function $\phi: D\subseteq \mathbb{R}^2 \to \mathbb{R}$ is \textit{harmonic} if it is twice continuously differentiable and satisfies Laplace's equation $\phi_{xx} + \phi_{yy} = 0$. A function $f=u+iv: D\subseteq\mathbb{C} \to \mathbb{C}$ is \textit{complex-valued harmonic} if both $u$ and $v$ are harmonic in D. On a simply connected domain $D \subseteq \mathbb{C}$, such an $f$ can be written as $f = h+\overline{g}$ where both $h$ and $g$ are analytic.

Harmonic functions have regions where they are sense-preserving as well regions where they are sense-reversing. From a zero-counting perspective, these regions are important because the sign of the order of a zero depends on the region in which it is located. 
Sense-preserving regions are those where the Jacobian $J_f(z) = |h'(z)|^2 - |g'(z)|^2$, is positive  whereas sense-reversing regions are those where the Jacobian is negative. 

For a given complex-valued harmonic function $f$, the dilatation function of $f$, defined by $\omega(z) = g'(z) / h'(z)$ can be interpreted as a measure of how far $f$ is from being analytic. Because $J_f(z) = |h'(z)|^2 - |g'(z)|^2 = |h'(z)|^2(1 - |\omega(z)|^2)$, when $|\omega(z)|<1$, $f$ is sense-preserving. When $|\omega(z)|=1$, the Jacobian vanishes. This condition characterizes the set of non-isolated critical points, which form what is known as the critical curve.

\begin{definition}
The critical curve of a complex-valued harmonic function $f$ is $\{z \in \mathbb{C} : |\omega(z)| = 1\}$.
\end{definition}

Note that the critical curve separates the sense-preserving and sense-reversing regions of the domain and is thus crucial to a zero-counting strategy. Specifically, the region where $|\omega(z)|<1$ is sense-preserving, while the region where $|\omega(z)|>1$ is sense-reversing. 

The order of a zero of a harmonic function $f=h + 
\overline{g}$ is the exponent of the first non-vanishing term in the expansion of $f$ about the zero. In sense-preserving regions, the order is defined to be positive and in sense-reversing regions, the order is defined to be negative. If the point lies on the critical set, where $J_f(z)=0$, the zero is referred to as singular, and its order is undefined. For a more precise definition, see \cite{Duren_2004}.

A pole of a harmonic function $f$ is a point $z_0$ with the property that $f(z)$ approaches infinity as $z$ approaches $z_0$. Similarly to the order of a zero, the order of a pole is defined as the smallest exponent appearing in the local expansion of \( f \) about the singularity. The sign of the order, as with zeros, reflects the local orientation: if the function is sense-preserving near the pole, the order is taken to be positive; if sense-reversing, the order is negative. We refer the reader to \cite{suffridge2000local} for a more formal definition of these concepts.

Two useful theorems for counting zeros of meromorphic functions are the Argument Principle and Rouché's Theorem. While these theorems do not apply to complex-valued harmonic functions, we fortunately have the harmonic analogue of the Argument Principle proved by Duren \cite{duren1996argument} and extended to harmonic functions with poles Suffridge and Thompson \cite{suffridge2000local}:

\begin{theorem}[Argument Principle for Harmonic Functions with Poles]\label{argument principle}
    Let $f$ be harmonic, except for a finite number of poles, in a simply connected domain $D \subseteq \mathbb{C}.$ Let $\Gamma$ be a simple closed curve contained in D not passing through a pole or a zero. Suppose that $f$ has no singular zeros in D. Let $Z_{f,\Gamma}$ be the sum of the orders of the zeros of $f$ in $\Gamma$, and let $P_{f,\Gamma}$ be the sum of the orders of the poles of $f$ in $\Gamma$. Then

    \begin{align}
        \frac{1}{2\pi}\Delta_{\Gamma}\arg f(z) = Z_{f,\Gamma} - P_{f,\Gamma}.
    \end{align}
\end{theorem}

The quantity $\frac{1}{2\pi}\Delta_{\Gamma}\arg f(z)$ is called the \textit{winding number} of $f(\Gamma)$ about the origin. Throughout this paper, we refer to this quantity as $W_{f,\Gamma}$.

With this Argument Principle for Harmonic Functions, one can prove, in a manner similar to Brown and Churchill \cite{BC09}, the harmonic analogue of Rouché's Theorem.

\begin{theorem}[Rouché's Theorem for Harmonic Functions with Poles]\label{rouches principle}
    Suppose $p$ and $q$ satisfy the hypotheses for the Argument Principle for Harmonic Functions with Poles. If $p$ and $q$ are harmonic inside and on the simple closed contour $\Gamma$, if $|p(z)| > |q(z)|$ at each point on $\Gamma$, and if $p$ and $q$ have no poles on $\Gamma$ and no singular zeros inside $\Gamma$, then $Z_{p,\Gamma}-P_{p,\Gamma} = Z_{p+q,\Gamma}-P_{p+q,\Gamma}$.
\end{theorem}

In Section \ref{Formula For Number of Zeros}, we use Theorems \ref{argument principle} and \ref{rouches principle} to count the zeros, dependent only on $n$ and $W_{f_a,\Gamma}$. In Section \ref{ourcriticalcurve}, we observe that the image of the critical curve is a standard epicycloid that has undergone a transformation. Thus, we review the definition
of an epicycloid.

\begin{definition}\label{defepi}
    An epicycloid centered at the origin is the curve traced out by a fixed point on a circle of radius $r$ rolling on an origin-centered circle of radius $R$. The curve is given by the parametric equations

    \begin{align*}
        x(\theta) = (R+r)\cos\theta - r\cos\left(\frac{R+r}{r}\theta\right) \\
        y(\theta) = (R+r)\sin\theta - r\sin\left(\frac{R+r}{r}\theta\right)
    \end{align*}
\end{definition}

If the fraction $\frac{R}{r}$ is rational and written in lowest terms as $\frac{l}{k}$, then the curve has $l$ cusps and the rolling circle completes $k$ revolutions around the fixed circle in order to close the curve. For brevity, we refer to these revolutions of the rolling circle around the center of the epicycloid simply as revolutions.

\section{Our Critical Curve}\label{ourcriticalcurve}

In this section, we prove that the critical curve of \( f_a \) is the unit circle and that the image of the critical curve is the affine image of a standard epicycloid. We restrict to \( z \neq 0 \) because $f_a$ is not defined there. For such \( z \), we have:
\[
|\omega(z)| = \frac{|z^n + a z^{-n-1}|}{|a z^n + z^{-n-1}|} = \frac{|z^{2n+1} + a|}{|a z^{2n+1} + 1|}.
\]
Letting \( w = z^{2n+1} \), we obtain:
\[
|\omega(z)| = \left| \frac{w + a}{a w + 1} \right|.
\]

This is the modulus of a Möbius transformation applied to \( w \). The set of \( w \in \mathbb{C} \) for which this quantity is equal to 1 is the unit circle $|w| = 1$. Because $|w| = |z^{2n+1}| = |z|^{2n+1}$, this condition is equivalent to $|z|=1$. Therefore, the critical curve of \( f_a \), which we denote by \( \Gamma \), is:
\[
\Gamma = \{ z \in \mathbb{C} : |z| = 1 \} = \{ e^{i\theta} : \theta \in [0, 2\pi] \}.
\]

Note that $f_a$ is sense-reversing inside $\Gamma$ and sense-preserving outside $\Gamma$. We now analyze the image of the critical curve. Observe that $f_a(\Gamma)$ is parameterized by
\begin{align*}
f_a(e^{i\theta}) 
&= \frac{a}{n+1}e^{i(n+1)\theta} - \frac{1}{n}e^{-in\theta} 
   + \frac{1}{n+1}e^{-i(n+1)\theta} - \frac{a}{n}e^{in\theta} - 1 \\
&= \frac{a}{n+1}\big( \cos((n+1)\theta) + i\sin((n+1)\theta) \big) 
   - \frac{1}{n}\big( \cos(-n\theta) + i\sin(-n\theta) \big) \\
&\quad + \frac{1}{n+1}\big( \cos(-(n+1)\theta) + i\sin(-(n+1)\theta) \big) 
   - \frac{a}{n}\big( \cos(n\theta) + i\sin(n\theta) \big) - 1.
\end{align*}
Splitting $f_a(e^{i\theta})$ into real and imaginary parts $u$ and $v$ and substituting $\theta = \frac{\phi}{n}$ gives

\begin{align}\label{imageofcrit1}
u(\phi) = -\left( \frac{a+1}{n}\cos\phi - \frac{a+1}{n+1}\cos\left(\frac{n+1}{n}\phi\right) \right) - 1
\end{align}

\begin{align}\label{imageofcrit2}
v(\phi) = -\left( \frac{a-1}{n}\sin\phi - \frac{a-1}{n+1}\sin\left(\frac{n+1}{n}\phi\right) \right) 
\end{align}
where now $\phi \in [0, 2n\pi]$. Observe,

\begin{align*}
\begin{bmatrix}
u(\phi) \\
v(\phi)
\end{bmatrix}
&=
\begin{bmatrix}
-\left( \frac{a+1}{n}\cos\phi - \frac{a+1}{n+1}\cos(\frac{n+1}{n}\phi) \right) - 1\\
-\left( \frac{a-1}{n}\sin\phi - \frac{a-1}{n+1}\sin(\frac{n+1}{n}\phi) \right) 
\end{bmatrix} \\
&=
\begin{bmatrix}
-1 & 0 \\
0 & \frac{2}{a+1}-1
\end{bmatrix}
\begin{bmatrix}
\frac{a+1}{n} \cos\phi - \frac{a+1}{n+1} \cos(\frac{n+1}{n}\phi) \\
\frac{a+1}{n} \sin\phi - \frac{a+1}{n+1} \sin(\frac{n+1}{n}\phi)
\end{bmatrix} 
 + \begin{bmatrix}
-1 \\
0 
\end{bmatrix}
\end{align*}
If
\begin{align*}
    \boldsymbol{A} = \begin{bmatrix}
-1 & 0 \\
0 & \frac{2}{a+1}-1 
\end{bmatrix} \hspace{1em}
    \boldsymbol{b} = \begin{bmatrix}
-1 \\
0 
\end{bmatrix}
\end{align*}
and
\begin{align}\label{base_epi1}
\tilde{u}(\phi) = \frac{a+1}{n} \cos\phi - \frac{a+1}{n+1} \cos\left(\frac{n+1}{n}\phi\right)
\end{align}

\begin{align}\label{base_epi2}
\tilde{v}(\phi) = \frac{a+1}{n} \sin\phi - \frac{a+1}{n+1} \sin\left(\frac{n+1}{n}\phi\right)
\end{align}
then
\begin{align}\label{affine}
    \begin{bmatrix}
u(\phi) \\
v(\phi)
\end{bmatrix} = \boldsymbol{A}\begin{bmatrix}
\tilde{u}(\phi) \\
\tilde{v}(\phi)
\end{bmatrix} + \boldsymbol{b}.
\end{align}

Therefore, the image of our critical curve is the affine image of a standard epicycloid. This standard epicycloid determines the important geometry of our curve and for the remainder of this paper, we refer to it by equations \eqref{base_epi1} and \eqref{base_epi2}, or simply as $E_a$. See Figure \ref{comparison} for a comparison of $f_a(\Gamma)$ and $E_a$.

\begin{figure}[H]
  \centering
  \subfigure[$f_a(\Gamma)$, with $n=4$, and $a=3$]{
    \raisebox{10mm}{\includegraphics[scale = 0.36]{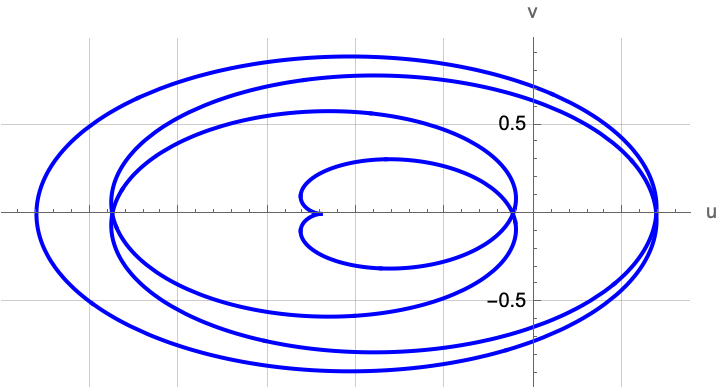}}
  } \hspace{5mm}
  \subfigure[$E_a$, with $n=4$, and $a=3$]{
    \includegraphics[scale = 0.34]{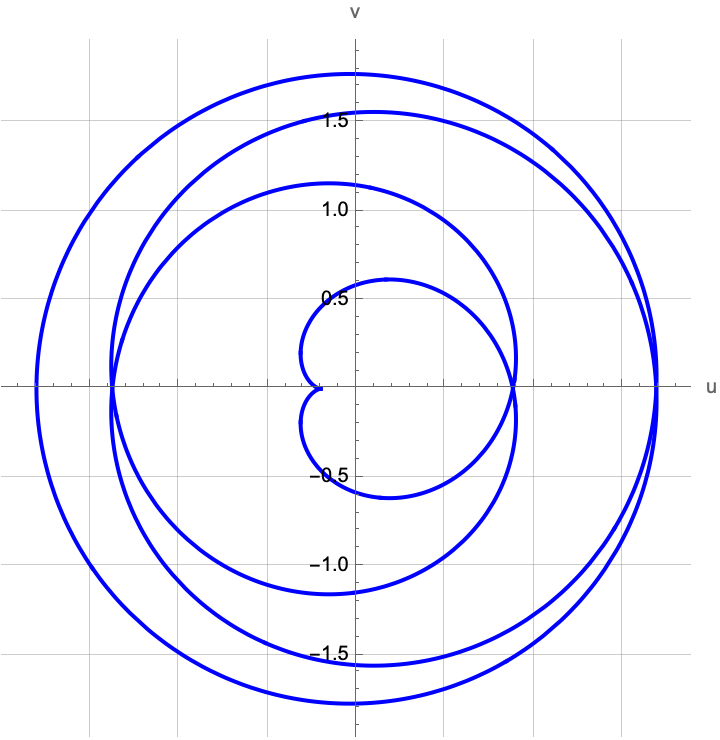}
  }
  \caption{The image $f_a(\Gamma)$ of our critical curve given in \eqref{imageofcrit1} and \eqref{imageofcrit2} is plotted on the left. On the right is the standard epicycloid $E_a$.}
  \label{comparison}
\end{figure}

\begin{remark}
    Note that $f_a(\Gamma)$ is reflected across the imaginary axis, scaled, and shifted compared to the standard epicycloid $E_a$. Because the properties of $f_a(\Gamma)$, namely its winding number, are determined by $E_a$, much of our analysis in Section \ref{geometry section} will consider $E_a$.
\end{remark}

\section{Formula for Number of Zeros}\label{Formula For Number of Zeros}

In this section, we use Theorems \ref{argument principle}, \ref{rouches principle}, and our analysis of the critical curve in Section \ref{ourcriticalcurve} to derive a formula for the total number of zeros of $f_a$.

\begin{lemma}\label{bigcircle}
    Let $W_{f_a,\Gamma}$ denote the winding number of $f_a(\Gamma)$ about the origin. The total number of zeros of $f_a$ in $\mathbb{C}$,  denoted $T_a$, is:
    $$T_a = 2(n-W_{f_a,\Gamma}) + 1.$$
\end{lemma}

\begin{proof}
Let $R$ be sufficiently large so that all zeros of $f_a$ lie in the interior of the circle \(\Gamma_R := \{ z \in \mathbb{C} : |z| = R \}\). Because \( a > 1 \), for \( R \) sufficiently large, \( \frac{a}{n+1}z^{n+1} \) is the dominant term of \( f_a \) on $\Gamma_R$. By a Rouché-type argument,
\[
n+1 = \left(Z_R^{+} - Z_R^{-}\right) - \left(P_R^{+} - P_R^{-}\right)
\]
where \( Z_{R}^{\pm} \) and \( P_{R}^{\pm} \) count with multiplicity the sense-preserving and sense-reversing zeros and poles of \( f_a \) inside \( \Gamma_R \), respectively. Because $f_a$ has only one pole at the origin, which is inside the sense-reversing region, we have \( P_R^{+} = 0 \) and \( P_R^{-} = n \). Thus,
\[
n+1 = Z_R^{+} - Z_R^{-} + n \quad \text{or, equivalently,} \quad 1 = Z_R^{+} - Z_R^{-}.
\]

Now consider the critical curve of $f_a$, defined as \( \Gamma := \{ z \in \mathbb{C} : |z| = 1 \} \). Again, the only pole is at the origin, so \( P_\Gamma^{+} = 0 \) and \( P_\Gamma^{-} = n \). However, because the interior of \( \Gamma \) is sense-reversing, we have \( Z_\Gamma^{+} = 0 \). Thus, by the Argument Principle for Harmonic Functions,
\[
W_{f_a,\Gamma} = -Z_\Gamma^{-} + n.
\]
Additionally, because only the interior of \( \Gamma \) is sense reversing, all zeros of negative order must reside in the interior of $\Gamma$. Thus, \[Z_R^{-} = Z_\Gamma^{-} = n - W_{f_a,\Gamma}\]
Therefore, the total number of zeros is:

\[T_a = Z_R^{+} + Z_R^{-} = 2Z_R^{-} + 1 = 2(n - W_{f_a,\Gamma}) + 1\]
\end{proof}

\section{Investigating $W_{f_a,\Gamma}$}\label{geometry section}

In the last section, we related the total number of zeros of $f_a$ to the winding number $W_{f_a,\Gamma}$ of $f_a(\Gamma)$. In Section \ref{ourcriticalcurve}, we showed that $f_a(\Gamma)$ is the image of a standard epicycloid whose geometry is determined by $n$ and $a$. In this section, we investigate how $W_{f_a,\Gamma}$ depends on $n$, $a$, and $f_a(\Gamma)$.

\subsection{Basic Geometry of $\boldsymbol{E_a}$
}\label{basicgeo}

First, we prove a lemma about the geometry of $f_a(\Gamma)$. 

\begin{lemma}\label{basicgeo}
Let the curve $f_a(\Gamma)$ be parameterized by equations \eqref{imageofcrit1} and \eqref{imageofcrit2}. Then $f_a(\Gamma)$
\begin{enumerate}
    \item has exactly one cusp.
    \item completes $n$ full revolutions about its center before terminating
    \item is traced in the counterclockwise direction.
\end{enumerate}
\end{lemma}

\begin{proof}

    By comparing equations \eqref{base_epi1} and \eqref{base_epi2} and Definition \ref{defepi}, one can show that the radii associated with the standard epicycloid $E_a$ are
    \begin{align}\label{radii}
        R = \frac{a+1}{n(n+1)} \hspace{1em}\text{and} \hspace{1em} r = \frac{a+1}{n+1}.
    \end{align}
Therefore, $\frac{R}{r} = \frac{1}{n}$. Thus, $E_a$ has 1 cusp, and the rolling circle completes $n$ revolutions before closing the curve. Because the transformation in equation \eqref{affine} is affine, $f_a(\Gamma)$ has the same properties.

Recall that $E_a$ is a standard epicycloid and is thus traced in the counterclockwise direction. Because $a >1$, we have $\frac{2}{a+1}-1 < 0$. Therefore, $\det({\boldsymbol{A})} > 0$ and $\boldsymbol{A}$ is orientation preserving. Thus, $f_a(\Gamma)$ is also traced out in the counterclockwise direction.
    
\end{proof}

\begin{remark}
    We note that if $a<1$, then $f_a(\Gamma)$ is traced out clockwise, which makes $W_{f_a,\Gamma}$ negative. For more information on this case, see Section \ref{small a section}.
\end{remark}

\subsection{Lemmas on the Winding Number} 

We now establish bounds on the parameter \( a \) that determine when the winding number \( W_{f_a, \Gamma} \) is equal to 0 or \( n \). These bounds define intervals in which the origin is guaranteed to lie completely outside or completely inside the curve \( f_a(\Gamma) \), respectively. Because an epicycloid is defined as the curve traced out by a circle rolling around another fixed circle, the curve must live in an annulus. We define the annulus in which the standard epicycloid $E_a$ is contained as 
\begin{align}
    \mathcal{C} = \{z \in \mathbb{C} : R \leq |z| \leq R+2r\} \label{annulus}
\end{align}
where $R$ is the radius of the fixed circle and $r$ the radius of the rolling circle denoted in \eqref{radii}. Using this annulus, we prove Lemmas \ref{outside} and \ref{inside}.

\begin{lemma}\label{outside}
For the curve $f_a(\Gamma)$ parameterized by \eqref{imageofcrit1} and \eqref{imageofcrit2}, there exists $a' > 1$ such that for all $1<a<a'$, the origin is completely outside of $f_a(\Gamma)$ so that $W_{f_a, \Gamma} = 0$.
\end{lemma}

\begin{proof}
    Consider \( a' = \frac{n(n+1)}{2n+1} - 1 \). A quick calculation shows that \( a' > 1 \) when \( n \geq 4 \), since
\[
\frac{n(n+1)}{2n+1} > 2 \quad \Longleftrightarrow \quad n^2 - 3n - 2 > 0,
\]
which holds for all \( n \geq 4 \). If $R = \frac{a+1}{n(n+1}$ and $r =\frac{a+1}{n}$, then
\begin{align*}
    R+2r = \frac{a+1}{n(n+1)} + 2\left(\frac{a+1}{n}\right) \\
    = (a+1)\left( \frac{1}{n(n+1)} + \frac{2n}{n(n+1)}\right) \\ 
    = (a+1)\frac{2n+1}{n(n+1)}
\end{align*}
which is less than $1$ if and only if 
$$a < \frac{n(n+1)}{2n+1} - 1 = a'$$

Thus, if $1<a<a'$, then $R + 2r < 1$. Because $E_a$ is a standard epicycloid centered at $z=0$, we see that $z=1$ lies \textit{outside} of the annulus $\mathcal{C}$. The matrix $\boldsymbol{A}$ reflects $E_a$ across the real axis but does not stretch it in this direction. Therefore, adding $\boldsymbol{b} = \begin{bmatrix}
-1 \\
0 
\end{bmatrix}$ results in the point $z=0$ lying fully outside the curve $f_a(\Gamma)$. Hence, $W_{f_a,\Gamma} = 0$ 
    \end{proof}

\begin{lemma}\label{inside}
    For the curve $f_a(\Gamma)$ parameterized by \eqref{imageofcrit1} and \eqref{imageofcrit2}, there exists $a'>1$ such that for all $a > a'$, $W_{f_a, \Gamma} = n$. 
\end{lemma}

\begin{proof}
    Consider $a' = n(n+1) - 1$. Note that $a' >1$ because $n \in \mathbb{N}$. Thus, if $a>a'$, 
    $$R = \frac{a+1}{n(n+1)} > \frac{a'+1}{n(n+1)} = 1$$. 
    Because $E_a$ is a standard epicycloid centered at $z=0$, we see that $z=1$ is within the inner disc of the annulus $A_{\Gamma}$. The matrix $\boldsymbol{A}$ reflects $E_a$ across the real axis but does not stretch it. Therefore, adding $\boldsymbol{b} = \begin{bmatrix}
-1 \\
0 
\end{bmatrix}$ results in the point $z=0$ lying fully within the curve $f_a(\Gamma)$. By Lemma \ref{basicgeo}, $f_a(\Gamma)$ is traced counterclockwise and has $n$ revolutions. Thus, $W_{f_a, \Gamma} = n$.
\end{proof}

Note that Lemmas \ref{outside} and \ref{inside} do not prove the existence of critical values of $a$ at which the winding number $W_{f_a,\Gamma}$ changes, but rather establish intervals in which $W_{f_a,\Gamma}$ is guaranteed to be $0$ and $n$. In Section \ref{endproof}, we use these lemmas together with the Intermediate Value Theorem to prove the existence of critical values of $a$. 

In the next subsection, we count the intersections of $f_a(\Gamma)$ with the real axis that occur on the right side of the curve.

\subsection{Counting Intersections with the Real Axis}

The curve $f_a(\Gamma)$ is centered at $z = -1$ and expands as $a$ increases to eventually fully envelope the origin. Each intersection of $f_a(\Gamma)$ with the real axis that occurs on its right side increases with $a$, eventually passing the origin and coinciding with an increase in the winding number. We prove the following lemma about how many of these unique intersections with the real axis exist.

\begin{lemma}\label{numberofuniqs}
    The curve $f_a(\Gamma)$ parameterized by \eqref{imageofcrit1} and \eqref{imageofcrit2} has $\lfloor\frac{n+1}{2}\rfloor$ unique intersections with the real axis on the right side of its center. Of these unique intersections,

    \begin{enumerate}
        \item if $n$ is even, then each is a double intersection, and
        \item if $n$ is odd, then each is a double intersection except for the outermost, which is a single intersection.
    \end{enumerate}
\end{lemma}

\begin{proof}
    This proof relies on the geometric behavior of the standard epicycloid $E_a$, which is the curve traced by a fixed point on a circle of radius $r$ rolling around a fixed circle of radius $R$. As the rolling circle completes its motion, the geometric argument of points on $E_a$ increases monotonically.
    
    By Lemma \ref{basicgeo}, the epicycloid $E_a$ is traced counterclockwise, has a single cusp, and completes $n$ revolutions. Because the curve has only one cusp, the radial distance from the origin increases monotonically during the first half of the parameter interval, reaches a maximum distance at $\phi = \pi n $, then decreases monotonically during the second half. Due to this symmetry, each intersection with the real axis is a \textit{double} intersection and is achieved twice through the parameter interval, except for the point at $\phi = \pi n$, which is passed through only \textit{once}. Substituting this point into the parameterization, we see
    \begin{align*}
        E_a(\pi n) = (-1)^n(R + 2r).
    \end{align*}
    Therefore, if $n$ is even, $E_a(\pi n) = R + 2r$, and this unique point of maximal distance is on the right side of the origin. If $n$ is odd, then $E_a(\pi n) = -(R + 2r)$, and the point lies to the left of the origin.

    Each intersection of $E_a$ with the real axis on its left side occurs when the geometric argument is an odd multiple of $\pi$ and is reached at least once within the first half of the parameter interval $\phi \in [0,\pi n]$. Therefore, the number of distinct intersections on the left side is equal to the number of odd integers less than or equal to $n$, which is $\lfloor \frac{n+1}{2} \rfloor$.
    
    Lastly, recall that the matrix \( \boldsymbol{A} \) reflects \( E_a \) across the imaginary axis. As a result, the curve \( f_a(\Gamma) \) has \( \left\lfloor \frac{n+1}{2} \right\rfloor \) distinct intersections with the real axis on its \textit{right} side. If \( n \) is even, each of these is a double intersection. If \( n \) is odd, all but the outermost point are double intersections, with the outermost point being the single intersection corresponding to the unique point of maximal distance from the center.

\end{proof}

\section{Completing the Proof of Theorem 1}\label{endproof}

Let $f_a$ be of the form \eqref{ourfamily}. By Lemma \ref{numberofuniqs}, $f_a(\Gamma)$ has $N = \lfloor\frac{n+1}{2}\rfloor$ unique intersections with the real axis on its right side. Because the standard epicycloid $E_a$ depends smoothly on $a$ and because $f_a(\Gamma)$ is the affine image of $E_a$, these intersections also depend smoothly on $a$. We may therefore label these intersections of $f_a(\Gamma)$ as real-valued continuous functions $x_j(a)$, where each $x_j(a)$ gives the location of the $j$-th intersection with the real axis on the right side of $f_a(\Gamma)$.

Recall that the standard epicycloid is contained in the annulus $\mathcal{C} = \{z \in \mathbb{C} :R \leq z \leq R+2r\}$ where    
\begin{align*}\label{radii}
        R = \frac{a+1}{n(n+1)} \hspace{1em}\text{and} \hspace{1em} r = \frac{a+1}{n+1}
\end{align*}

Because \( \mathcal{C} \) expands monotonically with \( a \), each real-axis intersection \( x_j(a) \) increases monotonically with \( a \). By Lemma~\ref{outside}, there exists \( a \), slightly greater than 1, such that the winding number is zero and we have \( x_j(a) < 0 \) for all \( j = 1, \dots, N \). By Lemma~\ref{inside}, for sufficiently large \( a \), the winding number is $n$ and we have \( x_j(a) > 0 \) for all \( j \). By the Intermediate Value Theorem, there exists a unique \( a_j \) such that \( x_j(a_j) = 0 \). That is, the \( j \)-th intersection of \( f_a(\Gamma) \) with the real axis passes through the origin at the critical value \(a_j \). We are now ready to count zeros in several cases:

\noindent
\textbf{Case 1:} $1 < a<a_1$. \\
\indent In this case, no intersections pass over the origin and $W_{f_a, \Gamma} = 0$. By Lemma~\ref{bigcircle}, the total number of zeros is 
\begin{align*}
    T_a = 2(n-0) + 1 = 2n+1.
\end{align*}
\textbf{Case 2:} $a_j < a< a_{j+1}$. \\
\indent In this case, $j$ points of intersection are to the right of the origin.

\noindent
\textbf{Subcase 2a:} $n$ is even. \\
\indent By Lemma \ref{numberofuniqs}, each intersection is a double intersection. Thus, each crossing of a double intersection with the origin corresponds to $W_{f_a,\Gamma}$ increasing by $2$. Therefore, $W_{f_a,\Gamma} = 2j$ and by Lemma~\ref{bigcircle}, the total number of zeros is 
\begin{align*}
    T_a = 2(n-2j) + 1 = 2n-4j +1.
\end{align*}
\textbf{Subcase 2b:} $n$ is odd. \\
\indent By Lemma \ref{numberofuniqs}, each intersection is a double intersection except for the outermost which is a single intersection. Therefore, $W_{f_a,\Gamma} = 2j -1$ and by Lemma~\ref{bigcircle}, the total number of zeros is  
\begin{align*}
    T_a = 2\left(n - (2j-1)\right) + 1 = 2n - 4j + 3.
\end{align*}
\textbf{Case 3:} $a>a_N$. \\
\indent In this case, all $N$ intersections are to the right of the origin and $W_{f_a, \Gamma} = n$. By Lemma~\ref{bigcircle}, the total number of zeros is 
\begin{align*}
    T_a = 2(n-n) + 1 = 1.
\end{align*}

\section{The Case $0<a<1$}\label{small a section}

Although our main analysis focuses on the case \( a > 1 \), the critical curve and its image remain the same for all \( a \neq 0 \). Despite the caustic tracing out clockwise when $a<1$ and negating the winding number, one could still easily apply Theorems \ref{argument principle} and \ref{rouches principle}. In this paper, however, we do not consider the case $0<a<1$ largely due to the lack of interesting behavior. To briefly illustrate, we claim that 
\[
R +2r = \frac{(a+1)(2n+1)}{n(n+1)} < 1
\quad \text{for } n \geq 4 \text{ and } 0 < a < 1.
\]
Because \( 0 < a < 1 \), we have \( 1 < a + 1 < 2 \), so it suffices to show
\[
\frac{2(2n+1)}{n(n+1)} < 1.
\]
One can easily verify that this holds for all \( n \geq 4 \), as it reduces to the inequality \( n^2 - 3n - 2 > 0 \), which is true in that range. Thus, the point $z=1$ never resides in the annulus $\mathcal{C}$. Therefore, the origin is always outside of the curve $f_a({\Gamma})$ and the number of zeros of $f_a$ never changes.

\section{Areas for further investigation}

This paper focuses on the case \( a \in \mathbb{R} \cap (1, \infty) \), but several directions for future research remain. One avenue is to explore what happens when \( a \) is complex, that is, \( a \in \mathbb{C} \setminus \{-1,1\} \), and whether the caustic remains analyzable. Additionally, while the transformation used in our paper was linear, it would be interesting to investigate other injective transformations that still produce caustics amenable to geometric analysis. In particular, what non-linear yet injective transformations yield caustics that can still be described and understood?

\bibliographystyle{plain}
\bibliography{Interesting_Caustics}

\end{document}